\documentclass{amsart}
\usepackage{amssymb,epsfig,newlfont,amsmath}

\usepackage[english]{babel}
\usepackage{color}
\newtheorem{theorem}{Theorem}
\newtheorem{lemma}[theorem]{Lemma}
\newtheorem{proposition}[theorem]{Proposition}

\numberwithin{equation}{section}

\newcommand {\RM} {\mathbb R}
\newcommand {\CM} {\mathbb C}
\newcommand {\ZM} {\mathbb Z}

\def\C{\mathbf C}
\def\G{\mathbf G}
\def\H{\mathbf H}
\def\L{\mathbf L}
\def\M{\mathbf M}
\def\U{\mathbf U}

\def\cp{\mathbf c}
\def\Go{G_0}
\def\btheta{\boldsymbol\theta}

\def\pni{\par\noindent}

\def\psstar{\psi^*}
\def\ws{w^*}
\def\ptf{\,\,.}

\title[ERRATUM]{ Erratum to \\ 
 ``On the cuspidal cohomology of $S$-arithmetic subgroups of reductive groups over number fields " \\
 by A. Borel, J.-P. Labesse, J. Schwermer \\
Compositio Mathematica, 102 (1996), 1-40.}

\author{Jean-Pierre Labesse}
\thanks{We thank Laurent Clozel and Rapha\"el Beuzart-Plessis for very useful
discussions as well as the referee for pertinent remarks.
The second named author gratefully acknowledges the support of the Max-Planck-Institute for Mathematics, 
Bonn, in 2020}

\address{Institut de Math\'ematique de Marseille, UMR 7373,
Aix-Marseille Universit\'e,
France}

\email{Jean-Pierre.Labesse@univ-amu.fr}

\author{Joachim Schwermer}
\address{Faculty of Mathematics, University Vienna, Oskar-Morgenstern-Platz 1, A-1090 Vienna, 
Austria  resp. Max-Planck-Institute
for Mathematics,
Vivatsgasse 7, D-53111 Bonn, Germany.}
\email{Joachim.Schwermer@univie.ac.at}

\begin{document}
\maketitle

\vspace{0.5cm}
Let $G$ be a connected semisimple algebraic group  defined over a number field $k$.
Consider the Lie goup $\G=G(k\otimes\RM)$ and let us denote by $\btheta$ a Cartan involution.
A $k$-automorphism $\phi$ of $G$ is said to be of Cartan-type if the automorphism
$\Phi$  induced by $\phi$ on $\G$ can be written $\Phi=\mathrm{Int}(x)\circ\btheta$
 where $\mathrm{Int}(x)$ is the inner automorphism defined by some $x\in \G$.
 Theorem 10.6 of \cite{BLS} establishes the following result regarding the existence of non-trivial 
cuspidal cohomology classes for $S$-arithmetic subgroups of  $G$:

\begin{theorem} \label{th1}
Let $G$ be an absolutely almost simple algebraic group defined over $k$ 
that admits a Cartan-type automorphism. When the coefficient system is trivial, 
the cuspidal cohomology of $G$ over $S$ does not vanish, that is, 
every $S$-arithmetic subgroup of $G$ has a subgroup of finite index 
with non-zero cuspidal cohomology with respect to the trivial coefficient system.
\end{theorem}
 
 The following assertion appears in  \cite{BLS} as Corollary 10.7: \par\medskip
{\sl Assume that $G$ is $k$-split and $k$ totally real or $G = {\mathrm{Res}}_{k'/k}G'$ where $k'$ is a CM-field. 
Then the cuspidal cohomology of $G$ over $S$ with respect to the trivial coefficient system does not vanish.}
 
 The proof of the Corollary amounts to exhibit, in each case,  a Cartan-type automorphism. 
 In the first case, dealing with split groups, the proof
 is correct. As regards the second case where $G = {\mathrm{Res}}_{k'/k}G'$ with $k'$ a CM-field
 it was observed by J. Rohlfs and L. Clozel independently that
the assertion (and the proof) must be corrected since, to make sense, 
the argument implicitly uses strong extra assumptions. 
First of all,  $G'$ has to be defined over $k$ so that the complex conjugation $c$
induced by the non trivial element $\sigma$ in $\mathrm{Gal}(k'/k)$
acts as a $k$-rational automorphism of $G$. Observe that
strictly speaking one has to extend the scalars from $k$ to $k'$ before applying the restriction functor.
Further assumptions are necessary 
so that Corollary 10.7 in \cite{BLS} should be replaced by:

\begin{theorem} \label{th2}
Let $k$ be a totally real number field and
 $G$ be an absolutely almost simple algebraic group defined over $k$. Consider cases
\pni (1) $G$ is $k$-split 
\pni (2)  $G = {\mathrm{Res}}_{k'/k}G'$ 
where $k'$ is a CM-field with $G'$ defined over $k$ totally real,
satisfying one of the following hypothesis 
\pni (2a)  $\H=G'(k\otimes\RM)$ has a compact Cartan subgroup.
\pni (2b)  $G'$ is split over $k$ and simply connected.
\pni
Then  the cuspidal cohomology of $G$ over $S$ with respect to the trivial coefficient 
system does not vanish.  
\end{theorem}

\begin{proof} In case (1), when $G$ is $k$-split, the proof is given in \cite{BLS}:
it relies on the first case of \cite[Corollary 10.7]{BLS} and
Theorem~\ref{th1}. In case (2a)
the result follows from Proposition~\ref{prop} below and Theorem~\ref{th1}.
In case (2b) the assertion  is a particular case of Theorem 4.7.1 of \cite{LBC},
which in turn relies on case (1) above.
\end{proof}

Most of the following proposition is well known
(see in particular  \cite[Corollary 2.9]{Sh},  \cite[Lemma 3.1]{Langl} and \cite[p.~2132]{A})
but not knowing of a convenient reference we sketch a proof.

Consider a connected semisimple algebraic group $\Go$ defined over $\RM$. We denote by
 $\G=\Go(\CM)$ the  group of its complex points and 
 by $\cp$ the complex conjugation on $\G$. Then $\H=\Go(\RM)$ is the group of fixed point under $\cp$. 
This anti-holomorphic involution is induced by an $\RM$-automorphism $c$ of 
$G_1=\mathrm {Res}_{\CM/\RM}\Go$.
Let $\btheta$ be a Cartan involution of $\G$. The group $\U$ of fixed points of $\btheta$ in $\G$
is a compact real form: $\U=U(\RM)$ where $U$ is a form of $\Go$ but not necessarily inner.
Let $G^{**}$  be the split outer form of $\Go$. Choose a splitting $(B^*,T^*, \{X_\alpha\}_{\alpha\in\Delta})$
 for $G^{**}$ over $\RM$
 where $T^*$ is a torus in a Borel subgroup $B^*$
 and  for each $\alpha\in\Delta$, 
 the set of simple roots, $X_\alpha$ is a root vector. 
Let $\psstar$ be the automorphism of $G^{**}$ that preserves the splitting and whose action 
defines $G^*$, the quasi-split inner form of $\Go$.
Let $w$ be the element of maximal length in the Weyl group for $T^*$.

\begin{proposition} \label{prop}
The following assertions are equivalent:
\pni(i) The automorphism $c$ is of Cartan-type.
\pni (ii)  The group $\U$ is an inner compact real form.
\pni(iii) The group $\H$ has a compact Cartan subgroup. 
\pni(iv) The group $\H$ admits discrete series.
\pni(v)  The involution $w\circ \psstar$  acts by $-1$ on the root system of $T^*$.
\end{proposition}
 
\begin{proof} The automorphism $c$ is of Cartan-type
 if, by definition,  $\cp=\mathrm{Int}(x)\circ\btheta$
  for some $x\in\G$ i.e. if and only if
 $U$ is an inner form of $\Go$ or equivalently of $G^*$. 
This proves the equivalence of (i) and (ii). 
Assume now that  $U$ is an inner form of $\Go$.
Up to conjugation under $\G$, we may assume that $\btheta$ 
 is of the form $\btheta=\mathrm{Int}(x)\circ \cp$
 with $x$ in the normalizer of a maximal torus $T$ in $\Go$ defined over $\RM$. 
In particular $x$ is semi-simple and its centralizer $\L$ in $\G$ is a complex
reductive subgroup of maximal rank.
The  cocycle relation $\mathrm{Int}(x\,\cp(x))=1$ implies that $\L$ is stable under $\cp$. Then $\cp$
 induces an anti-holomorphic involution on $\L$ whose fixed points  $\M=\L\cap\H=\U\cap\H$
 is a compact real form of $\L$.  A Cartan subgroup $\C$  of $\M$ is a compact Cartan subgroup in $\H$.
Hence (ii) implies (iii). 
Now, consider a torus $T$ in $\Go$ such that $\C=T(\RM)$ is compact. 
Then the complex conjugation $\cp$ acts by $-1$ on the weights of $T$. Hence
there is $n\in \mathbf G$ which belongs to the normalizer of $T^*\subset G^*$
such that $\mathrm{Int}(n)\circ \psstar$ acts as $-1$ on the root system of $T^*$. 
Now, since $\psstar$ preserves the set of positive roots,
$w=\mathrm{Int}(n)|_{T^*}$ is the element of maximal length in the Weyl group.
This shows that (iii) implies (v). 
The equivalence of (iii) and  (iv) is a well known theorem due to Harish-Chandra
(\cite{Har}, Theorem 13).  
Finally  Lemma~\ref{lem}  below shows that (v) implies (ii).
 \end{proof}

\begin{lemma} \label{lem}
Assume $w\circ \psstar$ acts by $-1$ on the root system.
Then, $G^*$ has an inner form  $U$ such that $\U=U(\RM)$ is compact.
\end{lemma}

   \begin{proof} Consider the complex Lie algebra $\mathfrak g=\mathrm{Lie\, }(\G)$. 
    Let $\Sigma$ be the set of roots, $\Sigma^+$ the set of positive roots
  and $\mathfrak g_\alpha$ the vector space attached to $\alpha\in\Sigma$
  with respect to the torus $T^*(\CM)$.
   Following Weyl \cite{W}, Chevalley \cite{Ch} and Tits \cite{Tits}
  one may choose elements  $X_\alpha\in\mathfrak g_\alpha$  
   for ${\alpha\in\Sigma}$, such that, if we define $H_\alpha\in\mathrm{Lie\, }(T^*(\CM))$ by
   $H_\alpha=[X_\alpha,X_{-\alpha}]$ we have
  $$ \alpha(H_\alpha)=2 \quad\hbox{and}\quad
   [X_\alpha,X_{\beta}]=N_{\alpha,\beta}X_{\alpha+\beta}  \quad\hbox{if $\alpha+\beta\in \Sigma$}
   \quad\hbox{ with }\quad 
 N_{\alpha,\beta}=-N_{-\alpha,-\beta}\in\ZM\ptf$$ 
 We assume the splitting compatible with this choice. Now let:
  $$Y_\alpha=X_\alpha-X_{-\alpha}\qquad Z_\alpha=i(X_\alpha+X_{-\alpha})
  \qquad W_\alpha=iH_\alpha\,\,.$$
The elements $Y_\alpha$ and $Z_\alpha$ for $\alpha\in\Sigma^+$ together with the
$W_\alpha$ for $\alpha\in\Delta$ build a basis for  a real Lie algebra $\mathfrak u$.
 As in the proof of  Theorem 6.3 in Chapter III of \cite{Hel}, we see that the Killing form 
 is negative definite on $\mathfrak u$ and hence the Lie subgroup $\mathbf U\subset\G$ 
 with Lie algebra $\mathfrak u$ is compact. 
  Since  $\psstar$  preserves the splitting 
 $\psstar(X_\alpha)=X_{\psstar(\alpha)}$  for $\alpha\in\Delta$.   
Let $w$ be the element of maximal length in the Weyl group for $T^*$.
There is an $n^*\in \G$, uniquely determined modulo the center, such that 
the inner auto\-morphism $\ws=\mathrm{Int}(n^*)$
acts as $w$ on $T^*$ and such that $\ws(X_\alpha)=-X_{w\alpha}$ for $\alpha\in\Delta$.    This 
auto\-morphism is of order 2 and commutes with $\psstar$.  
Now let $\phi=w^*\circ\psstar$. Since $w\circ \psstar$ acts by $-1$ on $\Sigma$
this implies $\phi(X_\alpha)=-X_{\phi(\alpha)}=-X_{-\alpha}$ for $\alpha\in\Delta\ptf$
It follows from the commutation relations and the relations
$N_{\alpha,\beta}=-N_{-\alpha,-\beta}$ that $\phi(X_\alpha)=-X_{-\alpha}$ and
$\phi(H_\alpha)=-H_{\alpha}$ for all $\alpha\in\Sigma\ptf$
Now $\phi$ which acts as an automorphism of the real Lie algebras $\mathfrak g^{**}$
generated by the $X_\alpha$ for $\alpha\in\Sigma$ can be extended to an antilinear involution of 
$\mathfrak g^{**}\otimes\CM=\mathfrak g=\mathfrak u+i\mathfrak u$.
This, in turn, induces a Cartan involution $\btheta$ on $\G$: 
its fixed point set is the compact group $\mathbf U=U(\RM)$ with Lie algebra $\mathfrak u$ and
 $U$ is the inner form of $G^*$ defined by the Galois cocycle $a_1=1$ and
$a_\sigma=\ws$.
\end{proof}
 
We observe that when, moreover,  $G^*$ is almost simple, which means that the root system of $G^*$ is irreducible,
the classification shows that  condition (v)  holds  except when  
$G^*$ is split of type $A_n$ with $n\ge2$, or $D_n$ with $n\ge3$ odd, or $E_6$ or when
$G^*$ is quasi-split but non split of type $D_n$ with $n\ge4$ even.

\bibliographystyle{amsalpha}

\end{document}